\documentclass{amsart}
\usepackage[utf8]{inputenc}

\usepackage[margin=1in]{geometry}
\usepackage[shortlabels]{enumitem}
\usepackage{amsmath,amscd,amsthm,amsfonts,amssymb}
\usepackage{eucal}
\usepackage{mathrsfs}
\usepackage{float}
\usepackage{tikz-cd, comment, fancyvrb}
\usepackage{todonotes}
\usepackage{xcolor}

\RequirePackage{color}
\definecolor{myred}{rgb}{0,0,0}
\definecolor{mygreen}{rgb}{0,0.5,0}
\definecolor{myblue}{rgb}{0,0,0.65}

\usepackage{hyperref}
\hypersetup{citecolor=blue}
\usepackage{tikz}
\usetikzlibrary{matrix,arrows,decorations.pathmorphing}

\theoremstyle{plain}
\newtheorem{theorem}{Theorem}[section]

\newtheorem{corollary}[theorem]{Corollary}

\newtheorem{lemma}[theorem]{Lemma}
\newtheorem*{lemma*}{Lemma}
\newtheorem*{proposition*}{Proposition}

\newtheorem*{truefact*}{Fact}

\theoremstyle{definition}

\theoremstyle{remark}
\newtheorem*{remark}{Remark}

\newcommand{\bb}[1]{\expandafter\newcommand\expandafter{\csname #1\endcsname}{{\mathbb {#1}}}} 

\bb C
\bb R
\bb Q
\bb Z
\bb N

\bb F 
\newcommand{\dx}{\mathrm{d}x}
\newcommand{\Irr}{\mathrm{Irr}}
\newcommand{\Gal}{\mathrm{Gal}}
\newcommand{\e}{\mathrm{e}}
\newcommand{\bchi}{\boldsymbol{\chi}} 
\renewcommand{\k}{\mathbf k}

\title{Lower bounds on weighted moments of primes in short intervals in number fields}
\author{R\'egis de la Bret\`eche}
\address{
Institut de Math\'ematiques de Jussieu-Paris Rive Gauche\\
Universit\'e de Paris, Sorbonne Universit\'e, CNRS UMR 7586\\
Case Postale 7012\\
F-75251 Paris CEDEX 13\\ France}
\email{regis.delabreteche@imj-prg.fr}

\author{Vivian Kuperberg}
\address{
Department of Mathematics \\ School of Mathematical Sciences \\ Tel Aviv University \\ P.O. Box 39040 \\ Ramat Aviv, Tel Aviv 69978 \\ Israel}
\email{vivkuperberg@tauex.tau.ac.il}
\thanks{The authors would like to thank Florent Jouve and Brad Rodgers for helpful comments. The second author is supported by the NSF Mathematical Sciences Research Program through the grant DMS-2202128.}
\date{}

\begin{document}

\maketitle

\begin{abstract}
 We consider an analog of a conjecture of Montgomery and Soundararajan on the moments of primes in short intervals in number fields; this analog was discussed and heuristically derived in a paper of the second author, Rodgers, and Roditty-Gershon. Adapting work of the first author and Fiorilli in the integer case, we establish lower bounds on a weighted version of these moments which agree with the conjectured values.
\end{abstract}

\section{Introduction}

Goldston and Montgomery \cite{MR1018376-goldston-montgomery} conjectured that for $\delta \in (0,1)$, as $X \to \infty$, the variance of the number of primes in short intervals of width $\delta X$ is given by
\begin{equation}\label{eq:goldston-montgomery-conjecture}
    \frac 1X \int_X^{2X} \big(\psi(x+\delta X) - \psi(x)  - \delta X\big)^2\mathrm{d}x \sim \delta X \log(\delta^{-1}),
\end{equation}
where $\psi(x) := \sum_{n \le x} \Lambda(n)$ is the Chebyshev function. This conjecture was extended by Montgomery and Soundararajan \cite{MR2104891-montgomery-sound} when they conjectured that the distribution of primes in short intervals is Gaussian. More precisely, they conjectured that for fixed $\varepsilon > 0$ and $n \in \mathbb N$, uniformly for $\frac{(\log X)^{1+\varepsilon}}{X} \le \delta \le \frac{1}{X^{\varepsilon}}$, the $n$th moment of this distribution is
\begin{equation*}
    \frac 1X \int_{X}^{2X} \frac{(\psi(x+\delta X)-\psi(x)-\delta X)^{n}}{X^{n/2}} \mathrm{d}x = (\mu_n + o(1))\big(\delta \log(\delta^{-1})\big)^{n/2},
\end{equation*}
where $\mu_n$ are the Gaussian moment constants: $\mu_n = 1 \cdot 3 \cdots (n-1)$ if $n$ is even and $\mu_n = 0$ if $n$ is odd. 

Montgomery and Soundararajan showed that their conjecture follows from a strong form of the Hardy--Littlewood $k$-tuples conjecture. By using similar, albeit heuristic, reasoning building off of Gross and Smith's~\cite{MR1763807-gross-smith-hardy-littlewood} generalization of the Hardy--Littlewood conjecture to number fields, the second author along with Rodgers and Roditty-Gershon \cite{MR4421937-kuperberg-rodgers-roditty-gershon} derived an analog of Goldston and Montgomery's conjecture in the number field setting. They consider a broad generalization of short intervals: for a fixed number field $K/\Q$, an embedding $m:K\rightarrow \mathbb R^n$, and a norm $\|\cdot\|$ on $\mathbb R^n$, for parameters $x\in \mathbb R^n$ and $H \in \mathbb R$, they consider the number of primes $\alpha \in K$ with $\|m(\alpha)-x\| \le H$. They conjecture that the variance of the number of primes in short intervals in $K$ is asymptotically proportional to $(1-\delta)$ times the expected number of primes, where~$H = X^\delta$ and $x$ varies over points $\|x\| \le X$; the conjecture for the variance in this setting is identical to the analogous conjecture in the integer setting.

{\color{black}  As a natural related question, one can consider prime ideals with norms lying in short intervals, that is to say, small sets of prime ideals determined by their norm. } In this setting, one might expect the analog of Goldston and Montgomery's conjecture to be
\begin{equation}\label{eq:goldston-montgomery-conjecture-in-number-fields}
    \frac 1X \int_X^{2X} \big(\psi_K(x+\delta X) - \psi_K(x) - \delta X\big)^2\mathrm{d}x \sim C_K \delta X \log (\delta^{-1}),
\end{equation}
where~$C_K$ is a constant depending only on $K/\Q$. This is identical to \eqref{eq:goldston-montgomery-conjecture} except that the prime-counting function $\psi_K(x)$ is given by $\psi_K(x) := \sum_{\substack{\mathfrak n \subset K \\ N_K(\mathfrak n) \le x}} \Lambda_K(\mathfrak n),$ where $\Lambda_K(\mathfrak n)$ are the coefficients of $-\frac{\zeta_K'(s)}{\zeta_K(s)}$ and where $\zeta_K(s)$ is the Dedekind zeta function.

In \cite{MR4322621-conjecture-Montgomery-Soundararajan-integers}, the first author and Fiorilli proved a lower bound, dependent on the Riemann Hypothesis, for a weighted version of Montgomery and Soundararajan's conjecture. For fixed $\kappa > 0$ and an even differentiable test function $\eta:\mathbb R \to \mathbb R$ such that $\eta(t),\eta'(t) \ll \e^{-\kappa t}$, they define 
\begin{equation*}
    \psi_{\eta}(x,\delta) := \sum_{n \ge 1} \frac{\Lambda(n)}{n^{1/2}} \eta\left(\delta^{-1} \log \left(\frac nx \right) \right).
\end{equation*}
{\color{black}We observe that if $\eta$ is chosen to be a smooth approximation of $\eta_0(t)=\e^{\delta t}1_{[-1,1]}$, we have
$$\psi_{\eta}(x,\delta)\approx \psi_{\eta_0}(x,\delta)=\frac{\psi(x\e^{\delta})-\psi(x\e^{-\delta})}{\sqrt{x} } ,$$
noting that the function $\eta_0$ itself does not satisfy our hypothesis.}

For $\delta > 0$ and a non-trivial even integrable function $\Phi: \mathbb R \to \mathbb R$ with $\Phi,$ $\widehat\Phi \ge 0$, they consider a weighted $n$-th moment
\begin{equation*}
    M_{n}(X,\delta;\eta,\Phi):=\frac 1{(\log X) \int_0^\infty \Phi} \int_1^\infty \Phi\left(\frac{\log x}{\log X}\right)\big(\psi_\eta(x,\delta) - x^{\tfrac 12} \delta \mathcal L_\eta(\tfrac{\delta}{2})\big)^n\frac{\mathrm{d}x}{x},
\end{equation*}
where $x^{\tfrac 12}\delta\mathcal L_\eta(\frac{\delta}{2})$ is an explicit function defined \eqref{eq:def-of-L-eta} to the expected main term of $\psi_\eta(x,\delta)$. In particular, they show for even moments that there exists $C_\eta>0$ such that 
\begin{align*}
M_{2n}(X,\delta;\eta,\Phi) \ge \mu_{2n} \delta^n \left(\alpha(\eta)\log(\delta^{-1}) + \beta(\eta)\right)^n& \left(1 + O_{\kappa,\eta} \left(\frac{n^2\delta}{\log(\delta^{-1} + 2)}\right)\right) 
+ O_\Phi\left(\delta\frac{(C_\eta \log(\delta^{-1} + 2))^{2n}}{\log X}\right).
\end{align*}
Here $\alpha$ and $\beta$ are defined as
\begin{equation}\label{def:alphabeta}
  \alpha(\eta)=\int_{\mathbb R} \widehat \eta^2(\xi){\rm d}\xi=\int_{\mathbb R}   \eta^2(\xi){\rm d}\xi,\qquad \beta(\eta)=\int_{\mathbb R} \widehat  \eta^2(\xi)\log |\xi|{\rm d}\xi.  
\end{equation} where,  for $\eta \in \mathcal L^1(\mathbb R)$, we recall the usual definition of the Fourier transform
$$\widehat \eta(\xi) := \int_{\R} \eta(t)\e^{-2\pi i \xi t}  {\rm d} t. $$

Their proof relies on a positivity argument in the explicit formula for the Riemann zeta function; since they consider an expansion with only positive terms, they can discard any ``off-diagonal'' terms without endangering the lower bound. For example, instead of assuming that the zeroes of $\zeta$ are simple, this technique allows one to freely discard any terms arising from possibly multiplicities of the zeroes.

The goal of this work is to extend the results of \cite{MR4322621-conjecture-Montgomery-Soundararajan-integers} to all extensions of $\mathbb Q$, just as the results of \cite{MR4421937-kuperberg-rodgers-roditty-gershon} extend the conjecture of Goldston and Montgomery. Many of the ideas are very similar to those of \cite{MR4322621-conjecture-Montgomery-Soundararajan-integers}, showing the flexibility of their technique. {\color{black}Despite the weighted averages $\psi_\eta(x,\delta)$ used in \cite{MR4322621-conjecture-Montgomery-Soundararajan-integers}, their results allow them to prove unconditional lower bounds on counts of primes in short intervals in integers. In the same way, our results on the weighted functions $\psi_{\eta,K}(x,\delta)$ imply unconditional lower bounds on counts of prime ideals with norms lying in a short interval. These derivations are identical to the arguments in \cite{MR4322621-conjecture-Montgomery-Soundararajan-integers}, so we omit them.}

Throughout, fix a finite Galois extension $K/\Q$. Denote by $\mathcal O_K$ the ring of integers of $K$, and for an ideal~$\mathfrak n \subset \mathcal O_K$, let $N_K \mathfrak n$ denote the absolute norm of the ideal $\mathfrak n$, so that $N_K\mathfrak n = [\mathcal O_K:\mathfrak n]$.

For fixed $\kappa > 0$, define $\mathcal E_\kappa \subset \mathcal L^1(\mathbb R)$ to be the set of differentiable even test functions $\eta:\R \to \R$ such that~$\widehat{\eta}(0) > 0$, such that for all $t \in \R$,
 \begin{equation*}
    \label{bound eta}
\eta(t),\eta'(t) \ll \e^{-\kappa|t|},\end{equation*}
and such that for all $\xi \in \R$,
\[0 \le \widehat{\eta}(\xi) \ll (|\xi|+1)^{-1}\log(|\xi|+2)^{-2-\kappa}.\]
This is the same class of test functions as defined in \cite{MR4322621-conjecture-Montgomery-Soundararajan-integers}. For $\eta \in \mathcal E_\kappa$ and $\delta < 2\kappa$, define
\[\psi_{\eta,K}(x,\delta):=\sum_{\substack{\mathfrak n \subset \mathcal O_K\\ \mathfrak n \text{ ideal}}} \frac{\Lambda_K(\mathfrak n)}{N_K\mathfrak n^{1/2}} \eta\Big(\delta^{-1}\log\Big(\frac {N_K\mathfrak n}x \Big)\Big),\]
where $\Lambda_K$ is defined by
$$\sum_{\substack{\mathfrak n \subset \mathcal O_K\\ \mathfrak n \text{ ideal}}}\frac{\Lambda_K(\mathfrak n)}{N_K\mathfrak n^s}=-\frac{\zeta'_K}{\zeta_K}(s).$$ 
{\color{black}In the case when $K=\Q$, the quantity $\psi_{\eta,K}(x,\delta)$ is precisely the quantity $\psi_\eta(x,\delta)$ defined in \cite{MR4322621-conjecture-Montgomery-Soundararajan-integers}.}

We consider the $n$th moment
\begin{equation}\label{eq:weighted-nth-moment-definition}
M_{n,K}(X,\delta;\eta,\Phi) := \frac 1{(\log X)\int_0^\infty \Phi}\int_1^\infty \Phi\Big(\frac{\log x}{\log X}\Big) \Big(\psi_{\eta,K}(x,\delta)-\delta x^{1/2}\mathcal L_{\eta}(\delta/2)+\delta\widehat \eta(0) {\rm ord}_{s=\frac 12}\zeta_K(s)\Big)^n \frac{\mathrm{d}x}{x},
\end{equation}
where $\mathcal L_\eta$ is the expected main term for $\psi_{\eta,K}(x,\delta)$, given by
\begin{equation}\label{eq:def-of-L-eta}
\int_0^\infty \frac{\eta(\delta^{-1}\log(\tfrac tx ))}{t^{\tfrac 12}} \mathrm{d}t = x^{\tfrac 12} \delta \int_{\mathbb R} \e^{\tfrac{\delta w}2} \eta(w)\mathrm{d}w =: x^{\tfrac 12}\delta \mathcal L_\eta(\tfrac \delta 2).
\end{equation}
The second subtracted term, $-\delta \hat{\eta}(0) \mathrm{ord}_{s=\tfrac 12}\zeta_K(s)$, is the contribution from a zero of $\zeta_K(s)$ at the central point~$s = \tfrac 12$, if any such zero exists. In many cases, it is expected that no such zero exists: for example, if the Galois group of $K/\mathbb Q$ is abelian, then the Dedekind zeta function is a product of Dirichlet $L$-functions, which are conjectured to never vanish at the central value (see for example \cite{MR2978845-bui-non-vanishing} where it is shown that 34.1\% of Dirichlet $L$-functions are non-vanishing at the central point). However, there are some extensions where~$\zeta_K(\tfrac 12) = 0$; see \cite{MR291122-armitage-central-zeroes} and \cite{MR4433141-kandhil} for a construction and further discussion. Vanishing of the Dedekind zeta function at the central point is a key counterexample to the philosophy that the nontrivial zeroes of~$\zeta_K(s)$ should be linearly independent, which is why we separate this contribution.

The function $\psi_{\eta,K}(x,\delta)$ morally counts powers of prime ideals $\mathfrak n$ with $N_K\mathfrak n$ lying in the interval $[x(1-O(\delta)),x(1+O(\delta))]$; for prime ideals in this interval, the weight $(N_K\mathfrak n)^{-1/2}$ is $x^{-1/2}(1+O(\delta))$. Based on~\eqref{eq:goldston-montgomery-conjecture-in-number-fields}, the analog of the conjecture of Goldston and Montgomery in number fields explored in  \cite{MR4421937-kuperberg-rodgers-roditty-gershon}, one might expect~$M_{2,K}(X,\delta;\eta,\Phi)$ to have a main term $\delta \log (\delta^{-1})$ (in terms of $\delta$), with some additional factors relating to the choice of test function $\eta$ and the field extension $K/\mathbb Q$. 

Before stating our main result, we introduce some notation relating to the field extension $K/\Q$. Let $n_K$ denote the degree $[K:\Q]$, let $\Delta_K$ denote the absolute discriminant of $K$, and let ${\rm rd}_K$ denote the root discriminant defined by ${\rm rd}_K := \Delta_K^{1/n_K}$.
For $j\geq 1$ and a finite group $G$, we define the character sum
\begin{equation}\label{def:lambda}
\lambda_j(G):=\sum_{\chi\in\Irr(G)}\chi(1)^j. 
\end{equation}
In our applications, $G$ will always be the Galois group $\Gal(K/\Q)$. For a Galois group $G = \Gal(K/\Q)$, the degree $n_K$ satisfies $n_K = \sum_{\chi\in\Irr(\Gal(K/\Q))} \chi(1)^2 = \lambda_2(G)$. Moreover, $\lambda_3(G)/{  n_K }\geq 1,$ and $\lambda_5(G) \le   \lambda_3(G)^2$.

We also define $q_j(K)$ and $z_j(K)$, two additional sums over characters of the Galois group, which will arise in our computations. For $\mathfrak f(K/\Q,\chi)$ the Artin conductor of $\chi$, let $q_j(K)$ and $z_j(K)$ be given by
\begin{equation}\label{def:qz}\begin{split}
 q_j(K)&:= \sum_{\chi\in\Irr(\Gal(K/\Q))}\chi(1)^{j-1} \log\big( \mathfrak f(K/\mathbb Q,\chi)2^{-2r_{2,K}\chi(1)}\big),\cr
 z(K,\chi)&:= {\rm ord}_{s=\frac 12}L(s,K,\chi),
 \cr z_j(K )&:= \sum_{\chi\in\Irr(\Gal(K/\Q))}\chi(1)^{j-1}z(K/\Q,\chi)=\sum_{\chi\in\Irr(\Gal(K/\Q))}\chi(1)^{j-1}{\rm ord}_{s=\frac 12}L(s,K/\Q,\chi).  
\end{split} 
\end{equation}
We will show in \eqref{encadrementj}
 that $   z_j(K)\ll q_j(K)\leq 2\lambda_j(G)  \log ({\rm rd}_K). $

With these definitions in hand, we are ready to state our main result.

\begin{theorem}\label{thm:main-theorem}
Let $K/\Q$ be a Galois extension of number fields with $G:=\Gal(K/\Q)$. Assume the Artin holomorphy conjecture and GRH. Let $0<\kappa<\tfrac 12$ and $\eta \in \mathcal E_\kappa$.
There exist positive constants $c_\eta$ and $C_\eta$, depending only on $\eta$, as well as a function $V(K;\eta,\delta)$ satisfying 
$$
V(K;\eta,\delta)  =  {\delta}\Big(
\lambda_{3}(G)\big(\alpha(\eta) \log(\delta^{-1})+ \beta(\eta)\big)+\alpha(\eta) q_{3}(K)-\alpha(\eta)z_{3}(K) \Big) +O_\eta\big( \lambda_{3}(G) \delta ^2 \big) $$ such that for $\delta\ll {\rm rd}_K^{-c_\eta}$ we have uniformly in $n\geq 1$
\begin{align*}
(-1)^nM_{n,K}(X, \delta; \eta,\Phi) \geq &   \mu_{n}V(K;\eta,\delta)^{n/2} \Big\{1+O\Big(\frac{n^2 n!\delta}{\log (1/\delta)}\frac{\lambda_5(G)}{ \lambda_3(G)^2}\Big)\Big\}\cr&+
O\Big(   V(K;\eta,\delta)^{n/2} \Big( \frac{C_{\eta}n_K^2 \log(\delta^{-1} )}{\lambda_3(G)} \Big)^{n/2}\frac{ \delta^{1-n/2} }{\log X} \Big).
\end{align*}
\par In particular, for  fixed $\eta$, $M_{2,K}(X, \delta; \eta,\Phi) \geq  V(K;\eta,\delta)(1+o(1))  $ whenever
$\log (1/\delta)=o(\lambda_3(G)(\log X) ).$ 
For fixed $m\geq 2$ and fixed $\eta$, 
$M_{2m,K}(X, \delta; \eta,\Phi) \geq \mu_{2m}V(K;\eta,\delta)^m (1+o(1)) $ whenever
\begin{equation}\label{eq:main-theorem-delta-bounds}
\delta^{-1 } =o\Big(\frac{\lambda_3(G)^{m/(m-1)}(\log X)^{1/(m-1)}}{  n_K^{2/(m-1)} (\log(\lambda_3(G)\log X))^{m/(m-1)} }\Big).
\end{equation} 
\end{theorem}

In the range when $\delta$ satisfies \eqref{eq:main-theorem-delta-bounds}, the lower bound $M_{2m,K}(X,\delta;\eta,\Phi) \geq \mu_{2m}V(K;\eta,\delta)^m(1+o(1))$ has a main term of size $(\delta \log(\delta^{-1}))^m$; this lower bound is consistent both with the Goldston--Montgomery conjectures in this setting and with the predictions of Montgomery and Soundararajan that the distribution is normal. This lower bound on the variance indicates that the constant $C_K$ in \eqref{eq:goldston-montgomery-conjecture-in-number-fields} is given by $  \lambda_3(G)$; the factor $\lambda_3(G)$ comes from the multiplicity of the zeroes of the Dedekind zeta function when the Galois group $G$ is nonabelian. In the case when $G$ is abelian, $\lambda_3(G) = n_K$, so the constant is simply the degree of the extension. However, the presence of a more complicated character sum in the nonabelian case is a notable aspect of Theorem \ref{thm:main-theorem}.
When $n=2m+1$ is odd, Theorem \ref{thm:main-theorem} shows that
\begin{equation*}
 M_{2m+1,K}(X,\delta;\eta,\Phi) \le O\left(V(K;\eta,\delta)^{n/2} \left(\frac{C_\eta \log(\delta^{-1})}{\lambda_3(G)}\right)^{n/2} \frac{\delta^{1-n/2}}{\log X}\right).
\end{equation*}

In particular, the upper bound for $|M_{2m+1,K}|^{1/(2m+1)}$ is smaller than the lower bound for $M_{2m,K}^{1/2m}$; this is consistent with the fact that the odd moments of a normal distribution are zero. However, Theorem \ref{thm:main-theorem} does not rule out the possibility that the odd moments are large and negative.

The condition \eqref{eq:main-theorem-delta-bounds} on the size of $\delta$ is fairly restrictive; Theorem \ref{thm:main-theorem} is meaningful when $\delta$ is reasonably close to $1$. In this range, the expected main term for $\psi_{\eta,K}(x,\delta)$ is large, so that Theorem \ref{thm:main-theorem} implies lower bounds not only for the asymptotic main term but also for second-order terms.

The proof of Theorem \ref{thm:main-theorem} follows closely the ideas of \cite{MR4322621-conjecture-Montgomery-Soundararajan-integers}. The key difference arising in this setting is that general Dedekind zeta functions can have zeroes with multiplicity. As we discuss further in 
Section \ref{sec:number-fields-background}, each zero of the Dedekind zeta function $\zeta_K(s)$ of a Galois extension $K/\Q$ is a zero of some Artin $L$-function of a character $\chi$ of the Galois group. Assuming Artin's holomorphy conjecture as well as independence and simplicity of zeroes for Artin $L$-functions, the multiplicity of each zero is expected to be precisely equal to the dimension of the character  $\chi$. A lower bound of the correct order of magnitude must therefore take into account the expected multiplicity of the zeroes of $\zeta_K(s)$. Thus, our estimates incorporate more involved combinatorial arguments in order to handle zeroes with higher multiplicity.

 Our methods generalize to the relative setting of Galois extensions $L/K$, where the base field is not necessarily $\mathbb Q$. However, for clarity, we state our arguments in the case of an extension of $\mathbb Q$. This setting has the advantage that $\psi_{\eta,K}$ is straightforwardly a weighted count of prime ideals in $K$.

\section{Background and notation for Dedekind zeta functions}\label{sec:number-fields-background}

For a Galois extension $L/K$ of number fields,
\begin{equation}\label{eq:Dedekind-zeta-decomposition}
    \zeta_L(s) = \zeta_K(s) \prod_{\chi} L(s,\chi,L/K)^{\chi(1)
    },
\end{equation}
where the product is taken over all nontrivial irreducible characters $\chi$ of the Galois group $\mathrm{Gal}(L/K)$, and $L(s,\chi,L/K)$ is defined for $\Re e(s)>1$ by the Euler product
\begin{equation}\label{eq:Galois-L-function-definition}
    L(s,\chi,L/K) := \prod_{\substack{\mathfrak p \lhd \mathcal O_K \\ \mathfrak p \text{ prime}}} L_{\mathfrak p}(s,\chi), \quad (L_{\mathfrak p}(s,\chi) := \det(\mathrm{Id} - N_K\mathfrak p^{-s} \rho(\phi_{\mathfrak p})|_{V^{I_{\mathfrak p}}}, \mathfrak p \lhd \mathcal O_K \text{ prime}).
\end{equation}
Here $V^{I_{\mathfrak p}}$ is the subspace of the representation space $V$ which is invariant under the inertia group $I_{\mathfrak p}$; see \cite{MR0218327-zeta-and-L-functions}.

Artin's holomorphy conjecture (see for example \cite[Chapter VII, Section 10]{MR1697859-neukirch}) states that each Artin $L$-function $L(s,\chi,L/K)$ for a nontrivial character $\chi$ can be extended to an entire function, and when $\chi$ is trivial, $L(s,\chi,L/K)$ is entire except for a simple pole at $s=1$. Assuming Artin's conjecture, formula \eqref{eq:Dedekind-zeta-decomposition} implies that the zeroes of $\zeta_L(s)$ correspond to zeroes of $\zeta_K(s)$ or zeroes of some $L(s,\chi,L/K)^{  \chi(1)}$, and the multiplicity of a zero (assuming independence and simplicity of zeroes) is governed by the dimension of the corresponding representation of $\mathrm{Gal}(L/K)$.

For an Artin $L$-function $L(s,\chi,L/K)$, define the completed Artin $L$-function of the character $\chi$ via
\begin{equation}\label{defLambda}
\begin{split}
    \Lambda(s,\chi,L/K) := \big(s(s-1)\big)^{E_0(\chi)}&\mathfrak f(L/K,\chi)^{s/2}\left( \prod_{\mathfrak p|\infty \text{ real}}\Gamma_{\mathbb R}(s)^{n_+}\Gamma_{\mathbb R}(s+1)^{n_-} \right)\\
    &\times\Gamma_{\mathbb C}(s)^{\chi(1)r_{2,K}}L(s,\chi,L/K),
\end{split}\end{equation}
where
\begin{equation}
    \label{defGamma}
    \Gamma_{\mathbb R}(s)  := \pi^{-s/2}\Gamma\left(\frac s2\right), \quad
    \Gamma_{\mathbb C}(s)  := 2(2\pi)^{-s}\Gamma(s),  
\end{equation}
$r_{1,K}$ is the number of real places of $K$, $r_{2,K}$ is the number of complex places of $K$, and if $\phi_{\mathfrak P}$ denotes the distinguished generator of $\Gal(L_{\mathfrak P}/K_{\mathfrak p})$, then $n_+(\mathfrak p) := \frac{\chi(1) + \chi(\phi_{\mathfrak P})}{2}$ and $n_-(\mathfrak p) := \frac{\chi(1)-\chi(\phi_{\mathfrak P})}{2}$. Note that $n_+(\mathfrak p)$ and $n_-(\mathfrak p)$ do not depend on the choice of $\mathfrak P$, although they do depend on $\mathfrak p$. Finally, $\mathfrak f(L/K,\chi)$ is the Artin conductor of $\chi$; we denote by $\mathfrak f(L/K,\chi)$ both the ideal which is the Artin conductor and its norm. The exponent $E_0(\chi)$ is $0$ if $\chi$ is nontrivial, and $1$ if $\chi$ is trivial.
 
The relative discriminant $\Delta_{L/K}$ and the Artin conductors of the character $\chi$ are related via the Conductor-discriminant formula
\begin{equation}\label{conductor discriminant}
\Delta_{L/K}=\prod_{\chi\in\Irr(\Gal(L /K ))} \mathfrak f(L/K,\chi)^{\chi(1)}
.
\end{equation}
In the case of an extension of $K/\Q$, this formula becomes
\begin{equation*}
\log\big(\Delta_{K}\big)=\sum_{\chi\in\Irr(\Gal(K/\Q ))} \chi(1)\log\big(\mathfrak f(K/\Q,\chi)\big),  
\end{equation*}
where $\Delta_K$ is the absolute discriminant.

Finally, formula \eqref{eq:Dedekind-zeta-decomposition} implies the following identity concerning the order of vanishing of $\zeta_K$ at the central point:
\begin{equation*}
{\rm ord}_{s=\frac 12}\zeta_K(s)=\sum_{\chi\in\Irr(\Gal(K/\mathbb Q  ))} \chi(1) \ {\rm ord}_{s=\frac 12}L(s,K/\Q,\chi).\end{equation*}

\section{Weil's explicit formula and auxiliary estimates}\label{sec:auxiliary-lemmas}

For $\delta_0 > 0$, define ${\mathcal F}(\delta_0)$ to be the set of measurable functions $F(x)$ satisfying
\begin{equation}\label{defPhiF}
\int_{-\infty}^\infty \e^{(\tfrac 12 + \delta_0)2\pi|x|}|F(x)|\mathrm{d}x < \infty,
\end{equation}
\begin{equation*}
\int_{-\infty}^\infty \e^{(\tfrac 12 + \delta_0)2\pi|x|}|\mathrm{d}F(x)| < \infty,
\end{equation*}  and such that $F(x) = \frac 12(F(x^{-}) + F(x^{+}))$ for all $x$ and $F(x) + F(-x) = 2F(0) + O(|x|)$.
For $F\in {\mathcal F}(\delta_0)$ with $\delta_0 > 0$, define
\begin{equation*}
\Phi_F(s) = \int_{-\infty}^\infty F(x) \e^{-(s-1/2)2\pi x}\mathrm{d}x,
\end{equation*}
where $-\delta_0 < \Re e(s) < 1 + \delta_0$.

\begin{lemma}\label{lem:weil-explicit-Galois-L-functions}
Assume Artin's holomorphy conjecture. Let $F\in {\mathcal F}(\delta_0)$ with $\delta_0 > 0$ and $\Phi_F$ defined in~\eqref{defPhiF}.    Let~$L/K$ be a Galois extension of number fields, let $\chi$ be an irreducible character of the Galois group~$\Gal(L/K)$, and define $L(s,\chi,L/K)$ by \eqref{eq:Galois-L-function-definition}. Assume that $K$ has $r_{1,K}$ real embeddings and $r_{2,K}$ complex embeddings. Then
\begin{align*}
\sum_{\gamma}\Phi_F(\rho) = &E_0(\chi)(\Phi_F(0) + \Phi_F(1)) +  \frac {F(0)}{\pi}\Bigg[\log \big(\mathfrak f(L/K,\chi)^{1/2}\pi^{-r_{1,K}\chi(1)/2}(2\pi)^{-r_{2,K}\chi(1)}\big) \\
&+r_{2,K} \chi(1) \frac{\Gamma'}{\Gamma}\Big(\frac 12\Big) +\frac 12 \sum_{\mathfrak p|\infty \text{ real}} \Big(n_+(\mathfrak p) \frac{\Gamma'}{\Gamma}\Big(\frac 14\Big) + n_-(\mathfrak p)\frac{\Gamma'}{\Gamma}\Big(\frac 34 \Big)\Big)\Bigg] \\
&-\frac 1{2\pi} \sum_{\mathfrak p \subset \mathcal O_K} \sum_{m=1}^\infty \frac{\log N_K\mathfrak p}{(N_K\mathfrak p)^{m/2}}\left(\chi(\sigma_\mathfrak p^m)F\left(\frac{-m\log N_K\mathfrak p}{2\pi}\right) + \bar{\chi}(\sigma_{\mathfrak p}^m)F\left(\frac{m\log N_K\mathfrak p}{2\pi}\right)\right) \\
&+ \int_0^\infty \big(2F(0)-F( x)-F(-x)\big)\Big( \sum_{\mathfrak p|\infty \text{ real}}\frac{n_+(\mathfrak p)\e^{-\pi x}+n_-(\mathfrak p)\e^{-3\pi x}}{1-\e^{-4\pi x}} 
+  r_{2,K}\frac{\chi(1)\e^{- \pi x }}{1-\e^{-{2\pi x}}}\Big)\mathrm{d}x,
\end{align*}
where the sum on the left-hand side is taken over all nontrivial zeroes of $L(s,\chi,L/K)$. 
\end{lemma}

Lemma \ref{lem:weil-explicit-Galois-L-functions} is more general than the same result for Dedekind zeta functions, stated in the following corollary.

\begin{corollary}[Weil's explicit formula for Dedekind zeta functions]\label{cor:weil-explicit-dedekind}
 Let $K/\Q$ be a number field of degree~$n_K$, norm $N_K$, and discriminant $\Delta_K$, and assume that $K$ has $r_{1,K}$ real embeddings and $r_{2,K}$ conjugate pairs of complex embeddings. Then

\begin{align*}
\sum_{\gamma}\Phi_F(\rho) = &\Phi_F(0) + \Phi_F(1) +  \frac {F(0)}{\pi}\Big( \log\big(\Delta_K^{1/2}\pi^{-r_{1,K}/2}(2\pi)^{-r_{2,K}}\big) + \frac{r_{1,K}}{2 } \frac{\Gamma'}{\Gamma}\Big(\frac14\Big) +  {r_{2,K}} \frac{\Gamma'}{\Gamma}\Big(\frac12\Big)\Big) \\
&-\frac 1{2\pi} \sum_{\mathfrak n \subset \mathcal O_K} \frac{\Lambda_K(\mathfrak n)}{(N_K\mathfrak n)^{1/2}} \Big(F\Big(\frac{-\log N_K\mathfrak n}{2\pi}\Big) + F\Big(\frac{\log N_K\mathfrak n}{2\pi}\Big)\Big) \\
&+ \int_0^\infty \big(2F(0)-F( x)-F(-x)\big)\Big( \frac{r_{1,K}\e^{-\pi x}}{1-\e^{-4\pi x}} 
+  \frac{ r_{2,K}\e^{- \pi x }}{1-\e^{-{2\pi x}}}\Big)\mathrm{d}x,
\end{align*}  
where the sum on the left-hand side is taken over all nontrivial zeroes of $\zeta_K$. 
\end{corollary}

\begin{remark}
Corollary \ref{cor:weil-explicit-dedekind} follows from Lemma \ref{lem:weil-explicit-Galois-L-functions} by taking $\chi$ to be the trivial character. However, the statement of Corollary \ref{cor:weil-explicit-dedekind} can also be obtained by noting that
\[\sum_{\substack{\gamma \\ \zeta_K(\tfrac 12 + i\gamma) = 0}} \Phi_F(\tfrac 12 + i\gamma) = \sum_{\chi} \chi(1) \sum_{\substack{\gamma \\ L(\tfrac 12 + i\gamma, \chi,K/\Q) = 0}} \Phi_F(\tfrac 12 + i\gamma), \]
where the outside sum on the right is taken over all characters $\chi$ of $\mathrm{Gal}(K/\Q)$, and then applying Lemma 
\ref{lem:weil-explicit-Galois-L-functions} to every term on the right. This latter formula simplifies to the statement of Corollary \ref{cor:weil-explicit-dedekind} by applying the Conductor-discriminant formula \eqref{conductor discriminant} and the fact that for a Galois extension $K/\Q$ (which must be either totally real or totally complex),
\[\sum_{\chi} \chi(1) n_+ = r_{1,K} + r_{2,K}, \qquad \sum_{\chi} \chi(1) n_- = r_{2,K}.\]
\end{remark}

\begin{proof}[Proof of Lemma \ref{lem:weil-explicit-Galois-L-functions}]
We begin with the functional equation for Artin $L$-functions, which can be found in \cite[Chapter VII]{MR1697859-neukirch}.  
Artin's holomorphy conjecture states that $\Lambda(s,\chi,L/K)$, defined in \eqref{defLambda}, is an entire function for all $\chi$. By \cite[Theorem 12.6]{MR1697859-neukirch}, 
\begin{equation*}
    \Lambda(s,\chi,L/K) = W(\chi)\Lambda(1-s,\bar{\chi},L/K),
\end{equation*}
where $W(\chi) \in \mathbb C$ is a constant with absolute value $1$. Then
\begin{equation*}
    \sum_{\gamma} \Phi_F(\rho) = \frac 1{2\pi i} \int_{(1+\delta_1)} \Big(\Phi_F(s) \frac{\Lambda'}{\Lambda}(s,\chi,L/K) + \Phi_F(1-s)\frac{\Lambda'}{\Lambda}(s,\bar{\chi},L/K) \Big)\mathrm{d}s,
\end{equation*}
since when moving the line of integration of $\Phi_F(s) \frac{\Lambda'}{\Lambda}(s,\chi,L/K)$ from $1 + \delta_1$ to $-\delta_1$, one picks up contributions from the zeroes (and can then obtain the integral on the right hand side by applying the functional equation for $\frac{\Lambda'}{\Lambda}(s,\chi,L/K)$).

The logarithmic derivative of $\Lambda(s,\chi,L/K)$ is given by
\begin{align*}
\frac{\Lambda'}{\Lambda}(s,\chi,L/K) = &E_0(\chi)\left(\frac 1s + \frac 1{s-1}\right) + \frac 12 \log \mathfrak f(L/K,\chi) - \frac{n_{K/\mathbb Q}\chi(1)}{2}\log \pi 
        -\chi(1)r_{2,K}\log 2\\
    &+ \frac 12 \sum_{\mathfrak p|\infty \text{ real}} \left(n_+(\mathfrak p) \frac{\Gamma'}{\Gamma}\left(\frac s2\right) + n_-(\mathfrak p) \frac{\Gamma'}{\Gamma}\left(\frac{s+1}{2}\right)\right) 
    + r_{2,K} \chi(1)\frac{\Gamma'}{\Gamma}(s) + \frac{L'}{L}(s,\chi,L/K).
\end{align*}
As shown in \cite[Proposition 2.3.1]{nathan-ng-thesis}, the logarithmic derivative $\frac{L'}{L}(s,\chi,L/K)$ is given by
\begin{equation*}
    \frac{L'}{L}(s,\chi,L/K) = -\sum_{\mathfrak p \in \mathcal O_K} \sum_{m=1}^\infty \frac{\chi(\sigma_{\mathfrak p}^m)\log N_K\mathfrak p}{(N_K\mathfrak p)^{ms}},
\end{equation*}
where $\sigma_{\mathfrak p}$ is the conjugacy class of the Frobenius elements. For $\mathfrak p$ unramified, $\chi(\sigma_{\mathfrak p}^m)$ is well-defined, and if $\mathfrak p$ is ramified, let $\mathfrak q|\mathfrak p$ and define $\chi(\sigma_{\mathfrak p}^m) = \frac 1{|I_{\mathfrak q}|}\sum_{\tau \in I_{\mathfrak q}} \chi(\sigma_{\mathfrak q}^m \tau).$

As in the proof of Theorem 12.13 in \cite{MR2378655-montgomery-vaughan},
\begin{align*}
  &\frac 1{2\pi i}\int_{(1+\delta_1)}\Big(\Phi_F(s) +\Phi_F(1-s)\Big)\frac{L'}{L}(s,\chi,L/K) \mathrm{d}s \\
   &= -\frac 1{2\pi} \sum_{\mathfrak p \in \mathcal O_K} \sum_{m=1}^\infty \frac{\log N_K\mathfrak p}{(N_K\mathfrak p)^{m/2}}\left(\chi(\sigma_\mathfrak p^m)F\left(\frac{-m\log N_K\mathfrak p}{2\pi}\right) + \bar{\chi}(\sigma_{\mathfrak p}^m)F\left(\frac{m\log N_K\mathfrak p}{2\pi}\right)\right).
\end{align*}

The remaining contribution to the integral can be computed via the residue theorem and applying Lemma~12.14 from~\cite{MR2378655-montgomery-vaughan}.
\end{proof}

\begin{lemma}\label{lemma explicit formula}
Assume Artin's holomorphy conjecture and GRH. Let $0<\kappa<\tfrac 12$ and $\eta \in \mathcal E_K$. For $t \ge 0$ and $0 < \delta < \kappa$,
\begin{equation*}
    \psi_{\eta,K}(\e^t,\delta)-\e^{\tfrac t2}\delta\mathcal L_{\eta}(\tfrac \delta{2})  +\delta\widehat \eta(0) {\rm ord}_{s=\frac 12}\zeta_K(s )= -\delta \sum_{\rho} \e^{(\rho-1/2)t}\widehat{\eta}\Big(\frac{\delta}{2\pi}\frac{\rho-1/2}{i}\Big) + O_{\kappa,\eta}\big(n_K E_{\kappa,\eta}(t,\delta)\big),
\end{equation*}
where $\rho\neq \tfrac12$ runs over the nontrivial zeroes of $\zeta_K(s)$, and
\begin{equation*}
    E_{\kappa,\eta}(t,\delta) := \begin{cases} \delta \e^{-  t/2} + \log(\delta^{-1} + 2)\e^{- {\kappa t}/{\delta}} &\text{ if } t \ge 1, \\ 
    \tfrac{\delta}{t} + \log(\delta^{-1} + 2)\e^{- {\kappa t}/{\delta}} &\text{ if }\delta \le t \le 1, \\
    \log(\delta^{-1} + 2) &\text{ if }0 \le t \le \delta.\end{cases}
\end{equation*}
\end{lemma}

\begin{proof} We follow the proof of Lemma 2.1 of \cite{MR4322621-conjecture-Montgomery-Soundararajan-integers}. 
Apply Corollary \ref{cor:weil-explicit-dedekind} with $F(u) := \eta\Big(\tfrac{t+2\pi u}{\delta}\Big)$, so that $\widehat{F}(\xi) = \e^{i\xi t}\frac{\delta}{2\pi} \widehat{\eta}\Big(\frac{\delta\xi}{2\pi}\Big).$ Then
\begin{equation*}
    \begin{split}
\psi_{\eta,K}(\e^t,\delta) &- \e^{t/2}\delta \mathcal L_{\eta}\Big(\frac{\delta}{2}\Big) = -\delta \sum_{\rho} \e^{(\rho -1/2)t} \widehat{\eta}\Big(\frac{\delta}{2\pi}\frac{\rho-1/2}{i}\Big) \\
&+ \e^{-t/2}\delta \int_{-\infty}^{\infty} \e^{\delta x/2}\eta(x)\dx - \sum_{\substack{\mathfrak n \subset \mathcal O_K \\ \mathfrak n \text{ ideal}}} \frac{\Lambda(\mathfrak n)}{N_K\mathfrak n^{1/2}} \eta\Big(\frac{t+\log N_K\mathfrak n}{\delta}\Big) \\
&+\eta\Big(\frac t{\delta}\Big)\Big(r_{1,K} \frac{\Gamma'}{\Gamma}\Big(\frac14\Big) + 2r_{2,K} \frac{\Gamma'}{\Gamma}\Big(\frac12\Big) + \log\big(|\Delta_K|\pi^{-r_{1,K}}(2\pi)^{-2r_{2,K}}\big)\Big) \\
&+ \int_0^\infty \Big(2\eta\Big(\frac t{\delta}\Big) - \eta\Big(\frac{t+x}{\delta}\Big) -\eta\Big(\frac{t-x}{\delta}\Big)\Big)\Big( \frac{r_{1,K}\e^{-x/2}}{1-\e^{-2x}} +  \frac{r_{2,K}\e^{-x/2}}{1-\e^{-x}} \Big) \dx.
\end{split}
\end{equation*} 
Performing the same careful analysis of the final integral as in the integer case, and noting that the final integral is bounded by the $n_K$th multiple of its integer analog, gives the result.
\end{proof}

 For $\chi\in\Irr({\Gal}( L/K  ))$,  $h \in \mathcal E_\kappa$, and $\delta > 0$, 
define
\begin{equation*}
b (\chi;h,\delta):=\sum_{\rho} 
 h\Big(\frac{\delta}{2\pi}\frac{\rho-\tfrac 12}{i}\Big) 
\end{equation*}
 where (assuming GRH) the sum on the left-hand side is taken over all nontrivial zeroes of $L(s,\chi,L/K)$ and 
\begin{equation}
   \label{def:b0hdelta}b_0 (\chi;h,\delta):=\sum_{\rho\neq 1/2} 
 h\Big(\frac{\delta}{2\pi}\frac{\rho-\tfrac 12}{i}\Big). 
\end{equation}
 
In Lemma \ref{lemma3.4}, we apply Lemma \ref{lem:weil-explicit-Galois-L-functions} in order to estimate $b(\chi;h,\delta).$ 

\begin{lemma}\label{lemma3.4} Let $L/K$ be a Galois extension of number fields, let $\chi$ be an irreducible character of the Galois group $\Gal(L/K)$. Assume the Artin holomorphy conjecture and GRH.
Let $0<\kappa<\tfrac 12$ and let $h:\R\to\R$ be a measurable function such that for all $\xi \in \R$, $0 \le h(\xi) \ll (|\xi|+1)^{-2}(\log(|\xi|+2))^{-2-\kappa}$, and for all $t \in \R$, $\widehat h(t), \widehat h'(t)\ll \e^{-\kappa|t|}$. Fix an Artin $L$-function $L(s,\chi,L/K)$ for the extension $L/K$.
For $0<\delta < 2\kappa$, 
\begin{align*}
\delta b (\chi;h,\delta)  = &
n_{K}\chi(1)\widehat{h}(0) \log(\delta^{-1})+n_{K}\chi(1)\int_{\mathbb R}h(\xi)\log |\xi|{\rm d}\xi +\widehat{h}(0)  \log\big( \mathfrak f(L/K,\chi)2^{-2r_{2,K}\chi(1)}\big)\\
&+O_h\big(n_{K}\chi(1)\delta\big).
\end{align*}
\end{lemma}

\begin{proof}[Proof of Lemma \ref{lemma3.4}]
Applying Lemma \ref{lem:weil-explicit-Galois-L-functions} with $F(x) = 2\pi \delta^{-1}\widehat{h}(-2\pi\delta^{-1}x)$ yields
\begin{equation*}
    \sum_{\rho} h\Big(\frac{\delta}{2\pi}\frac{\rho-\tfrac 12}{i}\Big) = \delta^{-1}\big(b_1(h) + b_2(h) + I(h)\big) + E_0(\chi)\Big(h\Big(\frac{\delta i }{4\pi}\Big) + h\Big(-\frac{\delta i}{4\pi}\Big)\Big),
\end{equation*}
where
\begin{align*}
    b_1(h) &:= \widehat{h}(0) \Big(\log\big(\mathfrak f(L/K,\chi)\pi^{-n_{K}\chi(1)}2^{-r_{2,K}\chi(1)}\big) \cr&\quad+ \sum_{\mathfrak p|\infty \text{ real}}\Big\{n_+(\mathfrak p) \frac{\Gamma'}{\Gamma}\Big(\frac14\Big)+n_-(\mathfrak p)\frac{\Gamma'}{\Gamma}\Big(\frac 34\Big)\Big\} + 2r_{2,K}\chi(1)\frac{\Gamma'}{\Gamma}\Big(\frac12\Big)\Big),
\cr
    b_2(h) &:= -\sum_{\mathfrak p \in \mathcal O_K}\sum_{m=1}^\infty \frac{\log N_{L/K}\mathfrak p}{N_{L/K}\mathfrak p^{m/2}} \big(\chi(\sigma_{\mathfrak p}^m)\widehat{h}(\delta^{-1}m\log N_{L/K}\mathfrak p) +\bar{\chi}(\sigma_{\mathfrak p}^m) \widehat{h}(-\delta^{-1}m\log N_{L/K}\mathfrak p)\big),
\end{align*}
and
\begin{equation*}
    I(h):=\int_0^\infty \big(2\widehat{h}(0) - \widehat{h}(\delta^{-1}x) - \widehat{h}(-\delta^{-1}x)\big)\Big(\sum_{\mathfrak p|\infty \text{ real}}\frac{n_+(\mathfrak p)\e^{-x/2}+n_-(\mathfrak p)\e^{-3x/2}}{1-\e^{-2x}} +  r_{2,K}\frac{\chi(1)\e^{-x/2}}{1-\e^{-x}}\Big) \dx.
\end{equation*}

Split the integral $I(h)$ into three ranges, $[0,\delta]$, $[\delta,1]$, and $[1,+\infty)$, and denote by $I_1(h)$, $I_2(h)$, and $I_3(h)$ the respective integrals. Then, using that $\widehat{h}(t) \ll e^{-\kappa|t|}, $
\begin{equation*}
    I_3(h) = \widehat{h}(0) \int_1^\infty  \Big(\sum_{\mathfrak p|\infty \text{ real}}\frac{2(n_+(\mathfrak p)\e^{-x/2}+n_-(\mathfrak p)\e^{-3x/2})}{1-\e^{-2x}} +  \frac{2r_{2,K}\chi(1)\e^{-x/2}}{1-\e^{-x}}\Big)\dx + O_h(n_{K}\e^{-\kappa/\delta}).
\end{equation*} 
Moreover,
\begin{equation*}\begin{split}
    I_2&(h)  =  n_{K}\chi(1)\widehat{h}(0) \log(\delta^{-1})\\
    \cr
    &+\widehat{h}(0) 
    \int_0^1 \Bigg(\sum_{\mathfrak p|\infty \text{ real}}\Big\{n_+(\mathfrak p) \Big(\frac{2\e^{-x/2} }{1-\e^{-2x}}-\frac1x \Big)  +n_-(\mathfrak p) 
    \Big(\frac{\e^{-3x/2}}{1-\e^{-2x}}- \frac1x   \Big)\Big\}+
    2r_{2,K}\chi(1)  \Big(\frac{\e^{- x/2 }}{1-\e^{- x}}-\frac1{x}\Big)\Bigg)\dx
    \cr&
-n_{K}\chi(1)\int_{\mathbb R}h(\xi)\int_1^\infty \cos (2\pi x\xi)\frac{\dx}{x}{\rm d}\xi+ O_h\big(n_{K}\chi(1)\delta \big)
\end{split}
\end{equation*}
and
\begin{equation*}\begin{split}
    I_1(h)  = & n_{K}\chi(1)\int_{\mathbb R}h(\xi)\int_0^1 \big(1-\cos (2\pi x\xi)\big)\frac{\dx}{x}{\rm d}\xi+ O_h\big(n_{K}\chi(1)\delta \big)
\end{split}\end{equation*}
with
$$n_{K}\chi(1)= \sum_{\mathfrak p|\infty \text{ real}}(n_+(\mathfrak p)+n_-(\mathfrak p))+2r_{2,K}\chi(1).$$

Gathering these calculations and using 
$$ \int_0^1 \big(1-\cos (2\pi x\xi)\big)\frac{\dx}{x}-\int_1^\infty \cos (2\pi x\xi)\frac{\dx}{x} 
=\log (\pi |\xi|)
+\int_0^1 \big(1-\cos (2  x )\big)\frac{\dx}{x}-\int_1^\infty \cos (2  x )\frac{\dx}{x} ,
$$ we get

\begin{align*}
\delta\sum_{\rho} h\Big(\frac{\delta}{2\pi}\frac{\rho-\tfrac 12}{i}\Big)  = &
n_{K}\chi(1)\widehat{h}(0) \log(\delta^{-1})+n_{K}\chi(1)\int_{\mathbb R}h(\xi)\log |\xi|{\rm d}\xi +\widehat{h}(0)  \log\big(\mathfrak f(L/K,\chi)2^{-r_{2,K}\chi(1)}\big)\\
&+\widehat{h}(0)\Big(\sum_{\mathfrak p|\infty \text{ real}}(n_+C_{1,+} + n_-(\mathfrak p)C_{1,-})+{2}r_{2,K}\chi(1)C_2\Big)+O_h\big(n_{K}\chi(1)\delta\big)
\end{align*}
with 
\begin{align*} 
C_{1,+} &:=   \int_0 ^1\Big( \frac{2\e^{-  x/2 }}{1-\e^{-2 x}}-\frac 1x\Big)\dx +\int_1^\infty \frac{2\e^{-  x/2 }}{1-\e^{-2 x}}\dx  + 
\frac{\Gamma'}{\Gamma}\Big(\frac 14\Big)   +  \int_0^1 (1-\cos (2  x ))\frac{\dx}{x}-\int_1^\infty \cos (2  x )\frac{\dx}{x}\\
C_{1,-} &:=   \int_0 ^1\Big( \frac{2\e^{-  3x/2 }}{1-\e^{-2 x}}-\frac 1x\Big)\dx +\int_1^\infty \frac{2\e^{-3x/2 }}{1-\e^{-2 x}}\dx  + 
\frac{\Gamma'}{\Gamma}\Big(\frac 34\Big)   +  \int_0^1 (1-\cos (2  x ))\frac{\dx}{x}-\int_1^\infty \cos (2  x )\frac{\dx}{x}\\
C_2 &:=   \int_0 ^1\Big( \frac{ \e^{-   x/2 }}{1-\e^{-  x}}-\frac 1x\Big)\dx +\int_1^\infty \frac{ \e^{-  x/2 }}{1-\e^{-  x}}\dx  +  
\frac{\Gamma'}{\Gamma}\Big(\frac 12\Big)   +  \int_0^1 (1-\cos (2  x ))\frac{\dx}{x}-\int_1^\infty \cos (2 x )\frac{\dx}{x}.
\end{align*}
In \cite{MR4322621-conjecture-Montgomery-Soundararajan-integers}, it was shown that $C_{1,+} = 0$. To calculate $C_{1,-}$, note that by \cite[\S II.0, Exercise 149]{MR3363366-tenenbaum}, 
\[\frac{\Gamma'}{\Gamma}\Big(\frac 34\Big)=
\int_0^\infty \Big(\frac{ \e^{-2x }}{x}- \frac{2\e^{-  3x/2 }}{1-\e^{-2 x}}\Big) \dx,\]
and thus
\[C_{1,-} = \int_0^\infty \frac{\e^{-2x}- \cos(2x)}{x} \dx = 0.\]
To calculate $C_2$, we use the  identity \cite[\S II.0, Exercise 149]{MR3363366-tenenbaum}
 $$ 
 \frac{\Gamma'}{\Gamma}\Big(\frac 12\Big)=
\int_0^\infty \Big(\frac{ \e^{-2x }}{x}-  \frac{2 \e^{-   x  }}{1-\e^{-2 x}}\Big) \dx 
=
\int_0^\infty \Big(\frac{ \e^{-2x }}{x}-  \frac{ \e^{-   x/2  }}{1-\e^{-  x}}\Big) \dx 
$$ 
and we get
\begin{align*}
C_2&=\int_0 ^1  \frac {\e^{-2x}-1}x \dx +\int_1^\infty \frac {\e^{-2x}}x  \dx +  \int_0^1 (1-\cos (2  x ))\frac{\dx}{x}-\int_1^\infty \cos (2  x )\frac{\dx}{x} 
 =\int_0 ^\infty  \frac {\e^{-2x}-\cos(2x)}x \dx=0.
 \end{align*}
The claimed estimate follows. 
\end{proof}

We can apply Lemma \ref{lemma3.4} to estimate the sums
\begin{equation}\label{defnuj}
    \nu_j(K;\eta,\delta):=\sum_{\chi\in\Irr(\Gal(K/\Q))}     \chi(1)^j b_0(\chi;|\widehat \eta|^2,\delta) .
\end{equation}
In order to do so, we recall the definitions \eqref{def:lambda}, \eqref{def:qz} and \eqref{def:alphabeta} concerning character sums.

\begin{lemma}\label{lemma:estnuj} Let $K/\Q$ be a Galois extension of number fields with $G:=\Gal(K/\Q)$. Assume the Artin holomorphy conjecture and GRH. 
    Let $j\geq 2,$ $\kappa > 0$,   $\eta\in\mathcal E_\kappa  $. We have 
  \begin{align*}
\nu_j(K;\eta,\delta)  = &\frac{1}{\delta}\Big(
\lambda_{j+1}(G)\big(\alpha(\eta) \log(\delta^{-1})+ \beta(\eta)\big)+\alpha(\eta) q_{j+1}(K)-\alpha(\eta)z_{j+1}(K) \Big) +O_\eta\big(\lambda_{j+1}(G) \big).
\end{align*}  
In particular, if $\delta\ll {\rm rd}_K^{-c}$ for suitable large constant $c>0,$ we have
\begin{align*}
\nu_j(K;\eta,\delta)  = &\frac{\lambda_{j+1}(G)\alpha(\eta)}{\delta}\Big(
  \log(\delta^{-1})+    O_\eta\big(\log ({\rm rd}_K)\big)\Big)\asymp
  \frac{\lambda_{j+1}(G)\alpha(\eta)}{\delta} 
  \log(\delta^{-1}) .
\end{align*} 
\end{lemma}

\begin{proof}
Following \cite[Proposition 5.21]{MR2061214-iwaniec-kowalski}, under GRH, we have
$$z(K/\Q,\chi)\ll \frac{\log \mathfrak f(K/\Q,\chi)}{\log (3\log \mathfrak f(K/\Q,\chi)/n_{K})}\ll \log \mathfrak f(K/\Q,\chi).$$
In fact, the bound  $z(K/\Q,\chi)\ll \log \mathfrak f(K/\Q,\chi) $ is unconditional  (see \cite[(5.27)]{MR2061214-iwaniec-kowalski}). For what follows, we will only need the unconditional bound.
Thus
$ z_j(K)\ll q_j(K) $.
Using   $\mathfrak f(L/K,\chi)\leq 2\chi(1)n_{K} \log ({\rm rd}_K) $  (see \cite[Lemma 4.1]{MR4400872-fiorilli-jouve-chebyshev}),
we get
\begin{equation}\label{encadrementj}
    z_j(K)\ll q_j(K)\leq 2\lambda_j(G) \log ({\rm rd}_K),
\end{equation}
where the factor of $n_{\mathbb Q}$ is simply $1$.
Now apply Lemma \ref{lemma3.4} for the extension $K/\Q$ (in place of $L/K$) and for $h=\widehat\eta^2$ with $\widehat{h}(0)=\alpha(\eta)$ to get
\begin{align*}b_0(\chi;|\widehat \eta|^2,\delta)&=\frac{1}{\delta}\Big(
\chi(1)\big(\alpha(\eta) \log(\delta^{-1})+ \beta(\eta)\big) +\alpha(\eta)  \log\big( \mathfrak f(K/\Q,\chi)\big)-\alpha(\eta)z(K/\Q,\chi)\Big)  
+O_h\big( \chi(1) \big).
\end{align*} 
Summing over $\chi\in G$ gives
  \begin{align*}
\nu_j(K;\eta,\delta)  = &\frac{1}{\delta}\Big(\lambda_{j+1}(G)\big(\alpha(\eta) \log(\delta^{-1})+ \beta(\eta)\big)+\alpha(\eta) q_{j+1}(K)-\alpha(\eta)z_{j+1}(K) \Big) +O_\eta\big(\lambda_{j+1}(G) \big).
\end{align*} 

This is precisely the first statement of Lemma \ref{lemma:estnuj}. The second statement follows by determining a range of $\delta$ for which the main term dominates. In particular, \eqref{encadrementj} bounds nearly all of the terms in our expansion of $\nu_j$ by $\log({\rm rd}_K)$, and the remaining terms are bounded in terms of $\delta$. Thus for any fixed $\eta$, if $\delta \ll {\rm rd}_K^{-c}$ for a suitable large constant $c > 0$, then
\begin{align*}
\nu_j(K;\eta,\delta)  = &\frac{ \lambda_{j+1}(G)\alpha(\eta)}{\delta}\Big(
  \log(\delta^{-1})+    O_\eta\big(\log ({\rm rd}_K)\big)\Big)\asymp
  \frac{ \lambda_{j+1}(G)\alpha(\eta)}{\delta} 
  \log(\delta^{-1}) .
\end{align*}

\end{proof}


\section{Proof of Theorem \ref{thm:main-theorem}}

Assume the Artin holomorphy conjecture and GRH. Beginning with \eqref{eq:weighted-nth-moment-definition}, apply Lemma \ref{lemma explicit formula}. Any zero $\rho $ can be written $\rho=\tfrac 12+i\gamma$. Since $\widehat \Phi$ and $  \Phi$  are even and  real-valued, 
\begin{align*}
 (-1)^n   M_{n,K}(&X, \delta; \eta,\Phi)=\frac {(-1)^n}{ \log X\int_0^\infty \Phi}\int_{0}^{\infty} \Phi\Big( \frac {t}{\log X}\Big)\big(\psi_\eta(\e^t,\delta)-\e^{\frac t2}\delta \mathcal L_\eta( \tfrac \delta 2 ) +\delta\widehat \eta(0) {\rm ord}_{s=\frac 12}\zeta_K(s,\chi)\big)^{n} {\rm d} t   \\
&=  \frac{\delta^n}{\int_0^\infty \Phi} \sum_{\gamma_1,\dots,\gamma_n\neq 0} \widehat \eta\Big(\frac{\delta\gamma_1}{2\pi}\Big)\cdots \widehat\eta\Big(\frac{\delta\gamma_n}{2\pi}\Big)\int_0^{\infty} \e^{i t  (\gamma_1+\dots+\gamma_n)\log X }\Phi(t) {\rm d} t +O\Big( \frac{ \delta (C_{\eta}n_K \log(\delta^{-1}+2))^n}{\log X }\Big)\\
&= \frac{ \delta^n}{2\int_0^\infty \Phi} \sum_{\gamma_1,\dots,\gamma_n\neq 0} \widehat \Phi\Big(\frac {  (\gamma_1+\dots+\gamma_n)\log X}{2\pi}\Big)\widehat \eta\Big(\frac{\delta\gamma_1}{2\pi}\Big)\cdots \widehat\eta\Big(\frac{\delta\gamma_n}{2\pi}\Big) +O\Big( \frac{ \delta (C_{\eta}n_K \log(\delta^{-1}+2))^n}{\log X} \Big).
\end{align*}

For even  $n=2m$, by positivity of $\widehat \Phi$ and $\widehat \eta$, we have
\begin{equation}
    \label{geqM2m}
     M_{2m,K}(X, \delta; \eta,\Phi)\geq   { \delta^{2m}} S_{2m} +O\Big( \frac{ \delta (C_{\eta}n_K \log(\delta^{-1}+2))^{2m}}{\log X} \Big),
\end{equation}
with
\begin{equation}\label{eq:S2m-definition}
S_{2m}:=\sum_{\substack{\gamma_1, \dots, \gamma_{2m}\neq 0 \\ \gamma_1 + \cdots + \gamma_{2m} = 0}} \widehat{\eta}\Big(\frac{\delta\gamma_1}{2\pi}\Big)\cdots \widehat{\eta}\Big(\frac{\delta\gamma_{2m}}{2\pi}\Big),
\end{equation}
whereas for odd $n=2m+1$, by positivity of $\widehat \Phi$ and $\widehat \eta$,
\begin{equation}
    \label{geqM2m+1}
      -M_{2m+1,K}(X, \delta; \eta,\Phi)\geq   O\Big( \frac{ \delta (C_{\eta} n_K\log(\delta^{-1}+2))^{n}}{\log X} \Big).
\end{equation}
 
 We shall prove the following lower bound to study the even case.
\begin{lemma}
\label{lemma S2mgeq} 
Let $K/\Q$ be a Galois extension of number fields with $G:=\Gal(K/\Q)$. Assume the Artin holomorphy conjecture and GRH. Let  $\eta\in\mathcal E_\kappa  $ fixed. Then, for  $m\geq 1$ and $\delta>0$
 $$S_{2m}\geq \mu_{2m}\nu_2(K;\eta,\delta)^m \Big(1+O_\eta\Big(m^2 m!\frac{\nu_4(K;\eta,\delta)}{  \nu_2(K;\eta,\delta)^2} \Big)\Big),$$ where  $S_{2m}$ is defined  by \eqref{eq:S2m-definition}. 
\end{lemma}

Its proof, which we present in Section \ref{sec:pfoflemmaS2mgeq} relies on an application of Lemma~5.5 of \cite{dfjouve-general-class-fcns}.

Applying Lemma \ref{lemma S2mgeq} to \eqref{geqM2m} gives
\begin{align*}
M_{2m,K}(X, \delta; \eta,\Phi)&\geq    \mu_{2m}\big(\delta^2\nu_2(K;\eta,\delta)\big)^m \Big(1+O\Big(m^2 m!\frac{\nu_4(K;\eta,\delta)}{  \nu_2(K;\eta,\delta)^2} \Big)\Big) +O\Big(n_K^{2m}\frac{ \delta (C'_{\eta} \log(\delta^{-1} ))^{2m}}{\log X} \Big).
\end{align*}
Lemma \ref{lemma:estnuj} implies that
when $\delta \leq {\rm rd}_K^{-c_\eta}$
$$\delta \nu_2(K;\eta,\delta)  \asymp
   { \lambda_{3}(G)\alpha(\eta)} 
  \log(\delta^{-1}),\qquad \frac{\nu_4(K;\eta,\delta)}{  \nu_2(K;\eta,\delta)^2}\asymp \frac{\delta}{\log (1/\delta)}\frac{\lambda_5(G)}{ \lambda_3(G)^2}
.$$
Thus,
\begin{align*}
M_{2m,K}(X, \delta; \eta,\Phi)&\geq    \big(\delta^2\nu_2(K;\eta,\delta))^m  \Big\{\mu_{2m}+O\Big(\frac{m^2 m!\mu_{2m}\delta}{\log (1/\delta)}\frac{\lambda_5(G)}{\lambda_3(G)^2} \Big)\Big\} +O\Big(\frac{ \delta (C'_{\eta}n_K  \log(\delta^{-1} ))^{2m}}{\log X} \Big)\cr
&\geq     V(K;\eta,\delta)^m  \Big\{\mu_{2m}+O\Big(\frac{m^2 m!\mu_{2m}\delta}{\log (1/\delta)}\frac{\lambda_5(G)}{\lambda_3(G)^2} +\Big( \frac{C_{\eta}n_K^2 \log(\delta^{-1} )}{\lambda_3(G)} \Big)^{ m}\frac{ \delta^{1-m} }{\log X} \Big)\Big\}
\end{align*}
where $C_{\eta}=(C'_{\eta})^2/\alpha(\eta)$ and $V(K;\eta,\delta):=\delta^2\nu_2(K;\eta,\delta).$ By Lemma \ref{lemma:estnuj}, this completes the proof of  Theorem~\ref{thm:main-theorem}.

\section{Proof of Lemma \ref{lemma S2mgeq}}\label{sec:pfoflemmaS2mgeq}

The following lemma is an application of lemma~5.5 of \cite{dfjouve-general-class-fcns}. It makes use of the classification of irreducible characters $\chi$ of $G$ according to their \emph{Frobenius--Schur indicator} $\varepsilon_2(\chi)$ (see 
for example \cite[Theorem 8.7]{MR1645304-Huppert}).

\begin{lemma}
\label{lemma pre combinatoire}
  For $\ell\in \mathbb N$, let $\eta\in \mathcal S_\delta$, $\psi\in \Irr(G)$, and let $\chi_1,\dots,\chi_{2\ell} \in \{ \psi, \overline{\psi} \}$. 
 If $\psi$ is unitary (that is, $\varepsilon_2(\psi)=0$) then there exists a constant $C_\eta$ such that we have the estimate
\begin{equation*}
 \sum_{\substack{\gamma_{\chi_1},...,\gamma_{\chi_{\ell }}> 0 \\
 \gamma_{\chi_{\ell+1}},...,\gamma_{\chi_{2\ell }}< 0 \\  \forall \gamma\in \R,\\ \# \{ k\leq 2\ell  : \chi_k\in \{ \psi, \overline{\psi} \} , \gamma_{\chi_k}= \gamma \} = \\ \# \{ k\leq 2\ell : \chi_k\in \{ \psi, \overline{\psi} \}, \gamma_{\chi_k}= -\gamma \}  }} \!\!\!\!\!\!\!\!\! \prod_{k=1}^{2\ell}\widehat \eta\Big(\frac{\delta \gamma_{\chi_k}}{2\pi}\Big)
\geq \max\Big\{
\ell! b_0(\psi;|\widehat \eta|^2,\delta)^\ell -C_\eta \ell!^2\ell(\ell-1)   b_0(\psi;|\widehat \eta|^2,\delta)^{\ell-1}  ,0\Big\},
\end{equation*}
where the $\gamma_{\chi_j}$ run through the multiset of imaginary parts of the zeros of $L(s,L/F,\psi)L(s,L/F,\overline\psi)$ (with multiplicity).

If $\psi$ is either orthogonal or symplectic (that is, $\varepsilon_2(\psi)\in \{ \pm 1\}$), there exists a constant $C_\eta$ such that 
$$ \sum_{\substack{\gamma_{1},\dots,\gamma_\ell>0 \\ \gamma'_{1},\dots, \gamma'_{\ell} < 0 \\  \forall \gamma\in \R,\\ \# \{ k\leq  \ell  :  \gamma_{ k}= \gamma \} = \\ \# \{ k\leq  \ell :  \gamma'_{ k}= -\gamma \}  }} \prod_{k=1}^{\ell}\widehat \eta\Big(\frac{\delta \gamma_{ k}}{2\pi}\Big)\widehat \eta\Big(\frac{\delta \gamma'_{ k}}{2\pi}\Big)
\geq \max\Big\{2^{-\ell}\ell! b_0(\psi;|\widehat \eta|^2,\delta)^\ell -C_\eta 2^{-\ell}\ell!^2\ell(\ell-1)    b_0(\psi;|\widehat \eta|^2,\delta)^{\ell-1} ,0\Big\},
$$
where the $\gamma_1,\dots,\gamma_\ell, \gamma'_1,\dots ,\gamma'_\ell$ run through the imaginary parts of the zeros of $L(s,L/F,\psi)$ (with multiplicity).
\end{lemma}

\begin{proof}[Proof of Lemma~\ref{lemma S2mgeq}]
Let 
$$\Irr(G)=\big\{ \psi_1,\psi_2,\ldots,\psi_{r_1},\psi_{r_1+1},\overline{\psi_{r_1+1}},\psi_{r_1+2}, \ldots,\psi_{r_1+r_2},\overline{\psi_{r_1+r_2}}\big\},$$
where $\psi_1,\dots \psi_{r_1}$ are real and $\psi_{r_1+1},\dots ,\psi_{r_1+r_2}$ are complex.  

Using the methods of \cite{dfjouve-general-class-fcns}, we prove a lower bound of the sum $S_{2m}.$
The main additional difficulties come from the fact that the zeroes of $\zeta_K(s)$ are not always simple. However, assuming the Artin holomorphy conjecture we can nevertheless find a lower bound by considering only ``expected'' multiplicities of the zeroes of $\zeta_K(s)$. To begin with,
$$S_{2m} =\sum_{\bchi=(\chi_1,\dots \chi_{2m})\in \Irr(G)^{2m}}\Big(\prod_{j=1}^{2m}  \chi_j(1)\Big)
\sum_{\substack{\gamma_1, \dots, \gamma_{2m}\neq 0 \\ \gamma_1 + \cdots + \gamma_{2m} = 0\\ L( 1/2+i\gamma_j,\chi_j,L/K)=0}} \widehat{\eta}\Big(\frac{\delta\gamma_1}{2\pi}\Big)\cdots \widehat{\eta}\Big(\frac{\delta\gamma_{2m}}{2\pi}\Big) .$$ 
In the inner sum the zeroes of $L(\tfrac 12+i\gamma_j,\chi_j,L/K)$ are counted with multiplicity, whereas in the definition of $S_{2m}$ the multiplicity was 
that of $L(\tfrac 12+i\gamma_j,\chi_j,L/K)^{\chi_j(1)}.$

Given a vector $\bchi=(\chi_1,\dots \chi_{2m}) \in \Irr (G)^{2m}$ and $1 \le j \le r_1+r_2$, define
$$ E_j(\bchi):= \big\{1\leq  k \leq 2m \colon \chi_k \in\{ \psi_j,\overline{\psi_j} \}  \big\}, $$
and define $\ell_j(\bchi):=|E_j(\bchi)|$ to be the size of $E_j(\bchi)$. Note that $ \sum_{j=1}^{r_1+r_2} \ell_j(\bchi) = 2m$.

All terms in the sum are positive, so we can get a lower bound by restricting the sum to certain tuples of characters and zeroes. In particular, we consider the sum over characters restricted to those $\bchi=(\chi_1,\dots ,\chi_{2m})$ that are elements of $\Irr (G)^{2m}$ and $(\gamma_{\chi_1},\dots,\gamma_{\chi_{2m}})$ which appear in conjugate pairs; that is, for which 
 for any  $j\leq r_1+r_2$ and $ \gamma \in \mathbb R$ we have
$$
\big|\big\{ k\in E_j(\bchi) \colon \chi_k\in\{\psi_j,\overline{\psi_j}\}, \gamma_{\chi_k}= \gamma \big\}\big| =\big|\big\{ k\in E_j(\bchi) \colon \chi_k\in\{\psi_j,\overline{\psi_j}\}, \gamma_{\chi_k}= -\gamma \big\}\big|.$$
Finally, we may further impose that $k_j(\bchi):=\tfrac12\ell_j(\bchi)  \in \mathbb N$, and restrict the sum over characters to the subset $\Irr_{2m}$ of vectors of characters $\bchi = (\chi_1,\dots,\chi_{2m}) \in  \Irr (G)^{2m}$ for which for every $r_1+1\leq j\leq r_1+r_2 $, $|\{ \ell \leq 2m : \chi_\ell = \psi_j\}|=|\{ \ell  \leq 2m: \chi_\ell = \overline{\psi_j}\}|$.

Now stratify the outside sum according to the values assumed by 
$k_j(\bchi)$. Given an $(r_1+r_2)$-tuple $\k=(k_1,\dots,k_{r_1+r_2}) \in \mathbb N^{r_1+r_2}$ such that $k_1+\dots+k_{r_1+r_2}=m$, it remains to evaluate the sum
$$
D(\k  ):=  \sum_{ \substack{\bchi=(\chi_1,...,\chi_{2m}) \in   \Irr(G)^{2m}\\ \forall j,\,  k_j(\bchi) =k_j  }} \Big(\prod_{j=1}^{2m}  \chi_j(1)\Big)\sum_{\substack{\gamma_{\chi_1},...,\gamma_{\chi_{2m}}\neq 0 \\ \forall j\leq r_1+r_2,\forall \gamma\in \R,\\ \# \{ k\in E_j(\bchi) : \chi_k{  \in\{\psi_j,\overline{\psi_j}\}}, \gamma_{\chi_k}= \gamma \} = \\ \# \{ k\in E_j(\bchi) : \chi_k{\in\{\psi_j,\overline{\psi_j}\}}, \gamma_{\chi_k}= -\gamma \}  }}  \prod_{j=1}^{2m}\widehat \eta\Big(\frac{\delta\gamma_{\chi_j}}{2\pi}\Big).
$$ 
After reindexing, $D(\k)$ becomes
\begin{equation}\label{Dk} 
 D(\k  )= \binom{2m}{2k_1,\dots,2k_{r_1+r_2}} \prod_{j=1}^{r_1+r_2}\big(  \psi_j(1)^{2k_j}\sigma_j(k_j,\delta)\big)
\end{equation}
with
$$\sigma_j(k_j,\delta):=\sum_{\substack{\gamma_{\chi_1},...,\gamma_{\chi_{2k_j}}\neq 0 \\ \forall \gamma\in \R,\\ \# \{ k\leq 2k_j  :  
\gamma_{\chi_k}= \gamma \} = \\ \# \{ k\leq 2k_j :  
\gamma_{\chi_k}= -\gamma \}  }} \!\!         \prod_{k=1}^{2k_j}\widehat \eta\Big(\frac{\delta\gamma_{\chi_k}}{2\pi}\Big)
=\binom{2k_j}{k_j}\sum_{\substack{\gamma_{\chi_1},\ldots,\gamma_{\chi_{k_j}}>0\\ \gamma_{\chi'_{1}},\ldots,\gamma_{\chi'_{ k_j}}<0\\ \forall \gamma\in \R_{>0},\\ \# \{ k\leq  k_j  :  
\gamma_{\chi_k}= \gamma \} = \\ \# \{ k\leq  k_j :  
\gamma_{\chi'_k}= -\gamma \}  }} \!\!         \prod_{k=1}^{ k_j}\widehat \eta\Big(\frac{\delta\gamma_{\chi_k}}{2\pi}\Big)\widehat \eta\Big(\frac{\delta\gamma_{\chi'_k}}{2\pi}\Big).$$

For~$j\geq r_1+1$ (i.e. $\psi_j$ is unitary), applying Lemma~\ref{lemma pre combinatoire} shows that
$$\sigma_j(k_j,\delta)
\geq  2^{k_j}\mu_{2k_j}  b_0(\psi_j;|\widehat \eta|^2,\delta)^{k_j}\max\Big\{ 1 -C_\eta \frac{{k_j}!{k_j}({k_j}-1)}{ b_0(\psi_j;|\widehat \eta|^2,\delta)}  ,0\Big\},
$$
since
$$\binom{2k_j}{k_j}k_j!=2^{k_j}\mu_{2k_j}.$$

If instead $j \le r_1$ (i.e. $\psi_j$ is either orthogonal or symplectic), then we may fix the sign of the imaginary parts $\gamma_{\chi_j}$ and deduce that
$$ \sigma_j(k_j)=\binom{2k_j}{k_j} \sum_{\substack{\gamma_{1},\ldots,\gamma_{k_j}>0\\ \gamma'_{1} ,\dots ,\gamma'_{k_j}< 0 \\  \forall \gamma\in \R,\\ \# \{ k\leq k_j  : \gamma_{k}= \gamma \} = \\ \# \{  k\leq  k_j : \gamma'_{k}= -\gamma \}  }}  \prod_{k=1}^{ k_j}\widehat \eta\Big(\frac{\delta\gamma_{\chi_k}}{2\pi}\Big)\widehat \eta\Big(\frac{\delta\gamma_{\chi'_k}}{2\pi}\Big).
$$
Applying Lemma~\ref{lemma pre combinatoire} once more yields the bound
$$\sigma_j(k_j)
\geq  \mu_{2k_j}b_0(\psi_j;|\widehat \eta|^2,\delta)^{k_j}\max\Big\{ 1 -C_\eta\ \frac{{k_j}!{k_j}({k_j}-1)}{ b_0(\psi_j;|\widehat \eta|^2,\delta)}  ,0\Big\}.
$$
We continue to follow proof of lemma 5.5 of \cite{dfjouve-general-class-fcns} reporting these estimates in \eqref{Dk} and using
$$
 \prod_{\ell=1}^{r_1+r_2} \max\big\{ 1-x_\ell    ,0\big\}
 \geq 1-\sum_{j=1}^{r_1+r_2}  x_j \qquad (x_\ell \geq 0).  $$ 
The main term is  equal to
$\mu_{2m}   \nu_2(K;\eta,\delta)^{m} , $
with $\nu_2(K;\eta,\delta)$ defined in \eqref{defnuj}
and the error term is
\begin{align*}
&\ll {m^2}m!\mu_{2m}  \Big(\sum_{j =1}^{r_1+r_2}  \psi_j(1)^4b_0(\psi_j;|\widehat \eta|^2,\delta)  \Big) \Big( \sum_{j=1}^{r_1}  \psi_j(1)^2 b_0(\psi_j;|\widehat \eta|^2,\delta) + 2\sum_{j =r_1}^{r_1+r_2}  \psi_j(1)^2 b_0(\psi_j;|\widehat \eta|^2,\delta) \Big)^{m-2} \\
&\ll\mu_{2m} \nu_2(K;\eta,\delta)^{m-2} {m^2}m! \nu_4(K;\eta,\delta)
. \end{align*}
with 
$  
\nu_4(K;\eta,\delta)$ defined in \eqref{defnuj}.
Then we have
$$S_{2m}\geq \mu_{2m}\nu_2(K;\eta,\delta)^m \Big(1+O\Big(m^2 m!\frac{\nu_4(K;\eta,\delta)}{  \nu_2(K;\eta,\delta)^2} \Big)\Big).$$
\end{proof}

\bibliographystyle{amsplain}
\bibliography{bibliography}

\providecommand{\bysame}{\leavevmode\hbox to3em{\hrulefill}\thinspace}
\providecommand{\MR}{\relax\ifhmode\unskip\space\fi MR }
\providecommand{\MRhref}[2]{%
  \href{http://www.ams.org/mathscinet-getitem?mr=#1}{#2}
}
\providecommand{\href}[2]{#2}
\begin{thebibliography}{10}

\bibitem{MR291122-armitage-central-zeroes}
John~V. Armitage, \emph{Zeta functions with a zero at {$s={1\over 2}$}},
  Invent. Math. \textbf{15} (1972), 199--205. \MR{291122}

\bibitem{MR2978845-bui-non-vanishing}
Hong~M. Bui, \emph{Non-vanishing of {D}irichlet {$L$}-functions at the central
  point}, Int. J. Number Theory \textbf{8} (2012), no.~8, 1855--1881.
  \MR{2978845}

\bibitem{MR4322621-conjecture-Montgomery-Soundararajan-integers}
R\'egis de~la Bret\`eche and Daniel Fiorilli, \emph{On a conjecture of
  {M}ontgomery and {S}oundararajan}, Math. Ann. \textbf{381} (2021), no.~1-2,
  575--591. \MR{4322621}

\bibitem{dfjouve-general-class-fcns}
R{\'e}gis de~la Bret{\`e}che, Daniel Fiorilli, and Florent Jouve, \emph{Moments
  in the {C}hebotarev density theorem: general class functions},
  arXiv:2301.12899, 2023.

\bibitem{MR4400872-fiorilli-jouve-chebyshev}
Daniel Fiorilli and Florent Jouve, \emph{Unconditional {C}hebyshev biases in
  number fields}, J. \'{E}c. polytech. Math. \textbf{9} (2022), 671--679.
  \MR{4400872}

\bibitem{MR1018376-goldston-montgomery}
Daniel~A. Goldston and Hugh~L. Montgomery, \emph{Pair correlation of zeros and
  primes in short intervals}, Analytic number theory and {D}iophantine problems
  ({S}tillwater, {OK}, 1984), Progr. Math., vol.~70, Birkh\"{a}user Boston,
  Boston, MA, 1987, pp.~183--203. \MR{1018376}

\bibitem{MR1763807-gross-smith-hardy-littlewood}
Robert Gross and John~H. Smith, \emph{A generalization of a conjecture of
  {H}ardy and {L}ittlewood to algebraic number fields}, Rocky Mountain J. Math.
  \textbf{30} (2000), no.~1, 195--215. \MR{1763807}

\bibitem{MR0218327-zeta-and-L-functions}
Hans Heilbronn, \emph{Zeta-functions and {$L$}-functions}, Algebraic {N}umber
  {T}heory ({P}roc. {I}nstructional {C}onf., {B}righton, 1965), Thompson,
  Washington, D.C., 1967, pp.~204--230. \MR{0218327}

\bibitem{MR1645304-Huppert}
Bertram Huppert, \emph{Character theory of finite groups}, De Gruyter
  Expositions in Mathematics, vol.~25, Walter de Gruyter \& Co., Berlin, 1998.
  \MR{1645304}

\bibitem{MR2061214-iwaniec-kowalski}
Henryk Iwaniec and Emmanuel Kowalski, \emph{Analytic number theory}, American
  Mathematical Society Colloquium Publications, vol.~53, American Mathematical
  Society, Providence, RI, 2004. \MR{2061214}

\bibitem{MR4433141-kandhil}
Neelam Kandhil, \emph{A note on {D}edekind zeta values at 1/2}, Int. J. Number
  Theory \textbf{18} (2022), no.~6, 1289--1299. \MR{4433141}

\bibitem{MR4421937-kuperberg-rodgers-roditty-gershon}
Vivian Kuperberg, Brad Rodgers, and Edva Roditty-Gershon, \emph{Sums of
  singular series and primes in short intervals in algebraic number fields},
  Ramanujan J. \textbf{58} (2022), no.~2, 291--317. \MR{4421937}

\bibitem{MR2104891-montgomery-sound}
Hugh~L. Montgomery and K.~Soundararajan, \emph{Primes in short intervals},
  Comm. Math. Phys. \textbf{252} (2004), no.~1-3, 589--617. \MR{2104891}

\bibitem{MR2378655-montgomery-vaughan}
Hugh~L. Montgomery and Robert~C. Vaughan, \emph{Multiplicative number theory.
  {I}. {C}lassical theory}, Cambridge Studies in Advanced Mathematics, vol.~97,
  Cambridge University Press, Cambridge, 2007. \MR{2378655}

\bibitem{MR1697859-neukirch}
J\"{u}rgen Neukirch, \emph{Algebraic number theory}, Grundlehren der
  mathematischen Wissenschaften [Fundamental Principles of Mathematical
  Sciences], vol. 322, Springer-Verlag, Berlin, 1999, Translated from the 1992
  German original and with a note by Norbert Schappacher, With a foreword by G.
  Harder. \MR{1697859}

\bibitem{nathan-ng-thesis}
Nathan Ng, \emph{Limiting distributions and zeroes of {Artin} {L}-functions},
  Ph.D. thesis, The University of British Columbia, The University of British
  Columbia, Vancouver, Canada, 2000.

\bibitem{MR3363366-tenenbaum}
G\'{e}rald Tenenbaum, \emph{Introduction to analytic and probabilistic number
  theory}, third ed., Graduate Studies in Mathematics, vol. 163, American
  Mathematical Society, Providence, RI, 2015, Translated from the 2008 French
  edition by Patrick D. F. Ion. \MR{3363366}

\end{thebibliography}

\end{document}